\documentclass{amsart}

\usepackage[stretch=10,shrink=10]{microtype}
\usepackage[showframe=false, margin = 1.5 in]{geometry}
\usepackage{mathtools}
\usepackage[shortlabels]{enumitem}
\usepackage{stmaryrd}
\usepackage{hyperref}
\hypersetup{
     colorlinks   = true,
     citecolor    = black
}
\usepackage{theoremref}

\makeatletter
\def\thmref@flush{%
   \ifx\thmref@last\empty\else
      \ifthmref@comma, \thmref@finaltrue\fi \thmref@commatrue
      \thmref@last \ifx\thmref@stack\empty\else s\fi \thmref@num 0
      \let\do\thmref@one \thmref@stack
      \ifcase\thmref@num\or\space and\else\thmref@finaltrue, and\fi
      ~\ref{\thmref@head}\let\thmref@stack\empty\fi}
\def\thmref@one#1{\ifnum\thmref@num>0,\fi
   \space\ref{#1}\advance\thmref@num 1\relax}
\makeatother

\newcommand{\E}{\mathbf{E}}

\renewcommand{\P}{\mathbf{P}}

\newcommand{\stationaryconfig}{\sigma}
\newcommand{\sprinkledconfig}{\configg}
\newcommand{\QuenchedP}{\mathrm{P}}
\newcommand{\1}{\mathbf{1}}
\newcommand{\xmid}{{x_{\mathsf{mid}}}}
\newcommand{\In}{\ii{1,n}}

\DeclareMathOperator{\Bin}{Bin}
\DeclareMathOperator{\Geo}{Geo}

\DeclarePairedDelimiter\abs{\lvert}{\rvert}%
\DeclarePairedDelimiter\norm{\lVert}{\rVert}%
\DeclarePairedDelimiter\tvnorm{\lVert}{\rVert_{\mathrm{TV}}}%
\DeclarePairedDelimiter\floor{\lfloor}{\rfloor}%
\DeclarePairedDelimiter\ceil{\lceil}{\rceil}%
\DeclarePairedDelimiter\ii{\llbracket}{\rrbracket}%

\newcommand{\eqd}{\overset{d}=}
\newcommand{\Z}{\mathbb{Z}}
\newcommand{\ZZ}{\mathbb{Z}}
\newcommand{\NN}{\mathbb{N}}
\newcommand{\N}{\mathbb{N}}
\newcommand{\s}{\mathfrak{s}}
\newcommand{\critical}{\rho}

\newcommand{\rt}[2]{#1^{\mathtt{right}}(#2)}
\newcommand{\lt}[2]{#1^{\mathtt{left}}(#2)}

\newcommand{\Left}{\ensuremath{\mathtt{left}}}
\newcommand{\Right}{\ensuremath{\mathtt{right}}}

\newtheorem{thm}{Theorem}[section]
\newtheorem{lemma}[thm]{Lemma}
\newtheorem{prop}[thm]{Proposition}
\newtheorem{cor}[thm]{Corollary}
\newtheorem{conjecture}[thm]{Conjecture}

\newtheorem{problem}[thm]{Problem}
\theoremstyle{remark}

\theoremstyle{definition}

\newcommand{\critFE}{\critical_{\mathtt{FE}}}

\newcommand{\restrictedOp}[2]{#1\if\relax\detokenize{#2}\relax\else^{#2}\fi}
\newcommand{\restrictedSfOp}[2]{\restrictedOp{\mathsf{#1}}{#2}}
\newcommand{\A}[1][]{\restrictedSfOp{A}{#1}}
\newcommand{\Stab}[1][]{\restrictedSfOp{S}{#1}}
\newcommand{\odom}{f}
\newcommand{\config}{\sigma}
\newcommand{\configg}{\tau}
\newcommand{\onepersite}{\mathbf{1}}
\makeatletter
\newcommand{\tmix}[1][\@nil]{%
\def\tmp{#1}%
   \ifx\tmp\@nnil
       \mathbf{t}_{\mathtt{mix}}
    \else
       \mathbf{t}_{\mathtt{mix}}^{(#1)}
    \fi}
\makeatother

\begin{document}

 \title[Stationary activated random walk]{Stationary activated random walk: regularity, heavy tails, and remixing}  

 \author{Tobias Johnson}
 \address{Tobias Johnson, Departments of Mathematics,  CUNY College of Staten Island}
\email{\texttt{tobias.johnson@csi.cuny.edu}}
 \author{Matthew Junge}
 	\address{Matthew Junge, Department of Mathematics, CUNY Baruch College}
	\email{\texttt{matthew.junge@baruch.cuny.edu}}
 \author{Josh Meisel}
 	\address{Josh Meisel, Department of Mathematics, CUNY Graduate Center and Bard College}
	\email{\texttt{jmeisel@bard.edu}}

\begin{abstract}
We study the stationary state of driven-dissipative activated random walk on an interval of length $n$ and prove that subintervals down to size $\log n$ have particle density near the critical value. We then deduce that the avalanche area is heavy-tailed and that the process forgets an initial stationary state after $o(n)$ steps.
We conjecture that the density fluctuations, avalanche area, and remixing time from stationarity are
governed by the same critical exponent. 
\end{abstract}

  \maketitle

\section{Introduction}

\emph{Activated random walk} (ARW) is a model of self-organized criticality \cite{dickman2010activated, BakTangWiesenfeld87}. It can be formulated as an interacting particle system on a graph with active and sleeping  particles. Active particles perform simple random walk at exponential rate 1. When an active particle is alone, it falls asleep at exponential rate $\lambda \in (0,\infty)$. Sleeping particles remain in place but become active if an active particle moves to their site. If all active particles fall asleep, the system stabilizes with every site in the final configuration either empty or containing exactly one sleeping particle. 
A configuration of particles on vertices $V$ is represented as a map
$\sigma\colon V\to\NN\cup\{\s\}$, where $\config(x) = \s$ indicates the presence of a single sleeping particle at $x$ and $\config(x) = k \in \N$ indicates the presence of $k$ active particles.

Most interest in ARW is for the process on the integer lattice $\mathbb Z^d$ \cite{rolla2020activated}. Unlike the abelian sandpile model, even the case $d=1$ is believed to exhibit all of the hallmarks of self-organized criticality, such as power-law statistics and a robust limiting critical density \cite{levine2023universality}. The \emph{fixed-energy} version of ARW takes place with a stationary ergodic initial configuration of active particles on $\mathbb Z^d$. This version is known to have a critical value $\critFE(d) = \critFE(d, \lambda)$. If the mean initial particle density is smaller than $\critFE(d)$, the process eventually fixates in any finite neighborhood, while if it is larger than $\critFE(d)$, any given neighborhood is visited infinitely often by active particles \cite{rolla2019universality}. This critical value is universal in that
it depends only on the density of the initial configuration. Throughout this article we will write $\critFE = \critFE(1)$.

The \emph{driven-dissipative} version of the process takes place on a finite box with sinks at the boundary and particles added one at a time. We focus on the one-dimensional version,
which takes place on the interval $\ii{1,n} := [1,n] \cap \mathbb Z$ with an additional edge from each boundary vertex $1$ and $n$ to a sink. 
The process is a Markov chain $(\sigma_t)_{t=0}^{\infty}$ on state space $\{0,\s\}^{\In}$, i.e., on
configurations of sleeping particles on $\In$. The initial configuration $\sigma_0\in \{0,\s\}^{\In}$
is given, as is the \emph{driving sequence} $\mathbf{u}=(u_t)_{t\geq 1}$ with $u_t\in\In$.
To advance the chain from step~$t$ to $t+1$, an active particle is added to $\sigma_t$ at
site $u_{t+1}$, waking up the sleeping particle if one is present there. The system is then \emph{stabilized} on $\In$, meaning the ARW dynamics are run, with particles removed upon reaching the sink, until the system settles on the stable configuration $\config_{t+1} \in \{0, \s\}^{\In}$. The driving sequence may be random or deterministic but is assumed to be independent of $\sigma_0$ and all particle dynamics. Two standard choices are \emph{uniform driving}, where each $u_t$ is sampled uniformly from $\In$, and \emph{central driving}, where $u_t = \lceil n/2 \rceil$ for all $t \ge 1$.
For convenience, we allow $t$ to be a noninteger by interpreting $\config_t$ as $\config_{\floor{t}}$.
Levine and Liang proved that the chain has a unique stationary distribution $\pi_n=\pi(n,\lambda)$ that does not depend on the
driving sequence \cite[Lemma~2.4]{levine2021exact}. We write $\P_{\pi_n}$ and $\E_{\pi_n}$ for the probability and expectation when sampling $\stationaryconfig$ from $\pi_n$.  
For any configuration $\sigma$, let $\abs{\sigma}$ denote its total number
of sleeping and active particles, and for any subinterval $I\subseteq\ii{1,n}$, let $\sigma(I)$ denote the restriction of $\sigma$  to $I$.

Driven-dissipative ARW is proposed as a model of self-organized criticality because
it naturally tends to its stationary distribution, which is believed to be critical in
various senses. In particular, the \emph{density conjecture} states that as the box size
tends to infinity, the stationary density tends to $\critFE$
\cite{dickman1998self,vespignani2000absorbing,DickmanRollaSidoravicius10}.
In our notation, the conjecture in dimension~$1$ is the statement that
$\lim_{n \to \infty} \E_{\pi_n}|\stationaryconfig|/n$ exists and equals $\critFE$.
Physicists have carried out extensive experimental studies of activated random walk 
and other related sandpile models; see \cite[Section~1.3]{rolla2020activated} and \cite{levine2023universality}.
Compared to related processes such as the abelian sandpile and stochastic sandpile models,
activated random walk is thought to be the most universal \cite{levine2023universality}.

The density conjecture has now been proven in dimension one \cite{hoffman2024density,forien2025newproof}.
In this paper, we further explore the properties of the stationary
distribution $\pi_n$ beyond its limiting density.
Our main results for a configuration $\sigma\sim\pi_n$ are:
\begin{description}
  \item[Regularity] The configuration $\sigma$ has density close to $\critFE$ not just on the 
    entire interval, but on all subintervals down to a logarithmic scale.
  \item[Heavy tails] When an active particle is added to $\sigma$, it provokes an avalanche whose
    area of activity is heavy-tailed.
  \item[Remixing] The addition of $o(n)$ active particles to $\sigma$ brings the system to a 
    fresh stationary state, independent of $\sigma$.
\end{description}
In different ways, these results indicate that $\pi_n$ is a robust and universal critical state.
And, as we discuss in Section~\ref{sec:exponents}, we believe that the strongest forms
of these results are related to a critical exponent for activated random walk.
We discuss the results in more detail now.

\subsection{Regularity}

The density conjecture in dimension one follows from the large deviation estimates \cite[Theorem 8.4 and Proposition 8.5]{hoffman2024density}, which together imply that for any $\epsilon > 0$,
\begin{align}
\P_{\pi_n}\Bigl(\abs[\big]{\,\critFE n - \abs{\stationaryconfig}\,} \geq \epsilon n\Bigr) \leq Ce^{-c n} \label{eq:regularity}
\end{align} 
where the constants $C,c>0$ are independent of $n$ but may depend on $\epsilon$ and $\lambda$. 
In words, the density of sleeping particles in the entire interval under the stationary distribution is exponentially unlikely to deviate from $\critFE$. 

The activated random walk critical state is believed to be highly uniform, in contrast with
the intricate but nonuniversal critical state of the abelian sandpile model 
(see \cite[Section~10]{levine2023universality}).
We confirm this by proving that a sample from $\pi_n$ has density close
to $\critFE$ on any subinterval down to logarithmic scale with high probability.

\begin{thm} \thlabel{thm:regularity}
For any interval $I\subseteq\In$ with $m=\abs{I}$ and any $\epsilon>0$,
\begin{align}
\P_{\pi_n}\Bigl(\abs[\big]{\,\critFE m - \abs{\stationaryconfig({I})}\,} \geq \epsilon m\Bigr) \leq Cn^6 e^{-c m} ,\label{eq:log-regularity}
\end{align}
for some constants $c,C>0$ depending only on $\epsilon$ and $\lambda$.
Thus there exists $B> 0$ depending on $\epsilon$ and $\lambda$ such that for any sequence of intervals $J_n \subseteq \In$ with $m_n=\abs{J_n}$ satisfying $\abs{J_n}  \geq B\log n$, it holds that $\P_{\pi_n}\bigl(\abs[\big]{\critFE m_n - \abs{\stationaryconfig({J_n})}} \geq \epsilon m_n\bigr)\to 0$ as $n\to\infty$.
\end{thm}
We will typically apply this result to show that density regularity holds simultaneously for all intervals of $\ii{1,n}$ of a given minimum length, by a union bound.

Although this result is not explicitly conjectured in Levine and Silvestri's survey 
\cite{levine2023universality}, it is in the same spirit as their Conjecture 2, ``Uniformity of the aggregate'' for the point-source model, and also their Conjecture 20 on ``Incompressibility.'' Our result does not address the stronger conjecture that $\pi_n$
is \emph{hyperuniform} \cite[Conjecture~26]{levine2023universality}, i.e., that the number
of particles in a sample of $\pi_n$ has sublinear variance.

Our proof proceeds by a simple but useful technique to transfer known facts about $\pi_n$ on
the entire interval $\ii{1,n}$ down to subintervals. Using the same technique,
we also prove an upper bound on the density of $\sigma$ on a subinterval 
that holds when $\sigma$ is the stabilization
of an arbitrary configuration rather than a sample from the stationary distribution; see 
\thref{prop:upper.regularity}.

\subsection{Heavy tails}
Power-law observables are typical of systems at criticality because of the lack of a characteristic scale.
Bak, Tang, and Wiesenfeld called power laws the  ``fingerprint'' of self-organized 
criticality \cite{BakTangWiesenfeld88}. 
Thus we expect power-law statistics in the stationary distribution of driven-dissipative ARW,
if it represents a self-driven critical state.
The following result confirms this for the \emph{avalanche area} $X_n'(\sigma)$, the number of sites that see activity
when centrally adding an active particle to a configuration of sleeping particles $\sigma\sim\pi_n$
and stabilizing.
In fact, we prove a heavy tail for the following statistic we call the \emph{avalanche extent},
which serves as a lower bound for the avalanche area:
\begin{align*}
  X_n(\sigma) = \max\Bigl\{\abs[\big]{v-\ceil{n/2}}\colon \text{activity occurs at site~$v$ when stabilizing $\sigma+\delta_{\ceil{n/2}}$
    on $\ii{1,n}$}\Bigr\}
\end{align*}
when $\sigma\sim\pi_n$. In words, we start with a configuration $\sigma$ of sleeping particles sampled from
the stationary distribution of the driven-dissipative chain. We add an active particle at the midpoint
of $\ii{1,n}$ and then define the avalanche extent $X_n(\sigma)$ as the maximum distance from the center 
visited in the course of stabilizing the configuration. We show that $\P_{\pi_n}(X_n(\sigma)\geq m)$
decays slower than $1/m$, so long as $m=\omega(\log n)$.

\begin{thm} \thlabel{thm:ht}
    For any sequence $(m_n)_{n \geq 1}$ satisfying  $m_n\leq n/2$ and $m_n/\log n\to\infty$,
    \begin{align}\label{eq:av-star}
      m_n\P_{\pi_n}\bigl( X_n(\sigma) \geq m_n\bigr)\to\infty
    \end{align}
   as $n\to\infty$.
\end{thm}


Two immediate corollaries of \thref{thm:ht} are that $\E_{\pi_n}[X_n(\sigma)] \to \infty$ as $n \to \infty$, and that the probability of an avalanche reaching one of the endpoints of $\In$ is $\omega(1/n)$.

For the abelian sandpile, the avalanche size has also been shown to be heavy-tailed, and in dimension
five and above the diameter of the avalanche has been shown to have critical exponent two
\cite{hanson2017inequalities,hutchcroft2020universality}. For ARW, the only previous result
on avalanches is a very simple argument of ours that
$\P_{\pi_n}( X_n(\sigma) \geq m)\geq 1/m$ for all $1\leq m\leq n/2$ \cite[Theorem~3]{hoffman2026exhibits}.
Our result here improves $1/m$ to $\omega(1/m)$ under the assumption $m=\omega(\log n)$. We expect that this tail is governed by a critical
exponent that can be related to other ARW critical exponents; see Section~\ref{sec:exponents}.

\subsection{Remixing}
Driven-dissipative ARW exhibits cutoff at $\critFE n$ under both uniform and central driving \cite{hoffman2025cutoff}. More specifically, denoting the total variation norm by $\tvnorm{\, \cdot \,}$, the mixing time
\begin{align*}
  \tmix[n](\epsilon) := \min\bigl\{t\colon \max_{\sigma_0}\tvnorm{\QuenchedP^{\sigma_0}(\sigma_t\in\cdot\:)-\pi_n}\leq\epsilon\bigr\}
\end{align*}
satisfies $\tmix[n](\epsilon)=\big(\critFE+o(1)\big)n$ for all $\epsilon \in (0,1)$. Here, the maximum is taken over all initial states $\sigma_0 \in \{0, \s\}^{\In}$, and
$\QuenchedP^{\sigma}$ denotes the law of the driven-dissipative chain $(\sigma_t)_{t\geq 0}$ starting
from $\sigma_0=\sigma$.

The gist of this result is that driven-dissipative ARW mixes as quickly as possible,
since from the empty configuration we must add at least $\critFE n$ particles to
bring the system to density $\critFE$, 
which by the density conjecture is the density of the chain at stationarity.
This can be seen as a universality result: the system's initial state is forgotten quickly.
In contrast, the abelian sandpile mixes slowly \cite{hough2019sandpiles,hough2021cutoff}, 
which is at the root of its nonuniversality \cite{levine2015threshold}.

In this paper, we investigate the robustness of this mixing result. Cutoff at $\critFE n$ represents
the worst-case behavior, starting from an empty configuration. Now we ask for the mixing time starting
from a sample from the stationary distribution, at which time we are already close to density $\critFE$.
To state this result, we first specify that when $\sigma$ is random, the measure $\QuenchedP^{\sigma}$
denotes the \emph{quenched law} of $(\sigma_t)_{t\geq 0}$, i.e., the random measure, measurable
with respect to $\sigma$, that gives the law of the chain conditional on $\sigma$.
The following
result shows that the mixing time starting from configuration $\sigma$ is $o(n)$
for all $\sigma$ outside of a set of configurations with small $\pi_n$-measure.
Recall that under $\E_{\pi_n}$, the random variable $\config$ is sampled from $\pi_n$.
\begin{thm}\thlabel{thm:remix}
For any $\epsilon>0$,
\begin{align}
\E_{\pi_n} \left\| \QuenchedP^{\stationaryconfig}(\sigma_{\epsilon n} \in \cdot \:) - \pi_n \right\|_{\mathrm{TV}} \leq Ce^{-c n} \label{eq:remix}
\end{align}
with uniform or central driving and $C,c>0$ constants that depend on $\lambda$ and $\epsilon$ but not $n$.
\end{thm}

Besides its relevance to universality, this mixing result is practically significant in describing
how long it takes to obtain multiple independent samples from $\pi_n$. It is also
reminiscent of the results obtained by Levine and Silvestri on IDLA on the cylinder 
\cite{levine2019long, silvestri2020internal}. More broadly, our result is in the spirit
of Markov chain results like \cite[Theorem~2]{lubetzky2017universality} that investigate 
non-worst-case mixing times.

\thref{thm:remix} could be extended in two directions. First, one could consider more general driving
than central and uniform. We mention, however, that in \thref{thm:ce} we prove that the time to remix from 
stationarity is 
$\Omega(n)$ under driving at the left boundary.
Second, one could consider mixing starting from a
general class of configurations with near-critical density rather than from stationarity.
See Section~\ref{sec:mixing.extensions} for further discussion.

\subsection{Critical exponents}\label{sec:exponents}
Each of our results has a (conjectural) critical exponent associated with it.
Regularity is governed by the \emph{fluctuation exponent} $\alpha$, the infimum
over all $\alpha'$ such that $\abs[\big]{\abs{\sigma} - \critFE n} \leq n^{\alpha'}$ with high probability
for $\sigma\sim\pi_n$. This exponent measures the combined
random fluctuation of $\abs{\sigma}$ from $\E\abs{\sigma}$ and the nonrandom
fluctuation of $\E\abs{\sigma}$ from $\critFE n$. The avalanche area $X_n'(\sigma)$ is expected to satisfy
$\P_{\pi_n}(X_n'(\sigma)>k)\asymp k^{-\beta}$ for some exponent $\beta$.
And the mixing time from stationarity with bulk driving should be on the order of $n^{\gamma}$ for some exponent
$\gamma$. Speaking very loosely---in particular, assuming that these exponents exist
in some precise way---our results indicate that they are all bounded by $1$.

\begin{conjecture}
  The critical exponents $\alpha$, $\beta$, and $\gamma$ can be defined rigorously
  and satisfy $\alpha=\beta=\gamma$.
\end{conjecture}

To give a heuristic argument for the conjecture, we start with
\cite[Theorem~1.2]{hoffman2025cutoff} which roughly speaking states that if
all sites on $\ii{1,n}$ are visited in the course of a stabilization yielding
configuration $\sigma$, then $\sigma$ is nearly distributed as $\pi_n$.
Now, suppose $\sigma\sim\pi_n$. The number of particles in $\sigma$
is at least $\critFE n - O(n^\alpha)$, and so the addition of $Cn^\alpha$
new particles is enough to make it supercritical, causing activity everywhere
and causing it to mix, suggesting that $\alpha=\gamma$. To relate $\beta$ to $\gamma$,
if $\P_{\pi_n}(X_n'(\sigma)\geq n)\approx n^{-\beta}$, then in $Cn^{\beta}$ steps of the chain 
it is likely that an avalanche occurs visiting everywhere, causing
the chain to mix.

The avalanche area exponent $\beta$ has been estimated as $0.259\pm 0.011$ on various
one-dimensional lattices for the stochastic sandpile model \cite{huynh2011abelian}.
While no published estimates are available for activated random walk, it is expected that the two
models fall in the same universality class \cite{DickmanRollaSidoravicius10}.
Rigorously determining these critical exponents is only a distant goal, 
but even a bound strictly below $1$ for any of these exponents would be a strong
result.

\subsection{Proof sketches}\label{sec:sketches}
\subsubsection*{Theorem~\ref{thm:regularity}}
We first note that we can obtain $\sigma\sim\pi_n$
as the stabilization of configuration $\onepersite_{\ii{1,n}}$ consisting of a single
active particle at each site of $\ii{1,n}$ \cite[Theorem~2.1]{levine2021exact}.
Our goal is to show that $\sigma$ has density close to $\critFE$ on $\ii{a,b}$.
To prove this, we imagine stabilizing $\ii{1,n}$ by first stabilizing $\ii{a,b}$.
Then, we continue stabilizing the configuration until a particle is sent back into $\ii{a,b}$, 
at which point we stabilize $\ii{a,b}$ again.
Let us call each of these stabilizations of $\ii{a,b}$ a \emph{left stabilization} or a
\emph{right stabilization} depending on whether it was provoked by a particle sent to $a$ or to $b$.
At some point---say after $L$ left stabilizations
and $R$ right stabilizations of $\ii{a,b}$---the entire system is stabilized.

After the initial stabilization of $\ii{a,b}$
and any deterministic number $\ell$ left stabilizations and $r$ right stabilizations,
the configuration on $\ii{a,b}$ is distributed as $\pi_{b-a+1}$, the stationary distribution
of the driven-dissipative chain on $\ii{a,b}$. Indeed, after the initial stabilization, the configuration
is stationary for the driven-dissipative chain on $\ii{a,b}$ by \cite[Theorem~2.1]{levine2021exact},
and further left and right stabilizations are simply steps in this chain, with driving either at $a$
or $b$ (which is irrelevant). Thus the density on $\ii{a,b}$ after $\ell$
left stabilizations and $r$ right stabilizations is close to $\critFE$
with high probability by \eqref{eq:regularity}, applied with $n$ replaced by $b-a+1$.

However, we are interested in the configuration on $\ii{a,b}$ after \emph{random} numbers $L$ and $R$ of 
left and right stabilizations, at which time the configuration on $\ii{a,b}$ is no longer distributed as
$\pi_{b-a+1}$. The trick to our proof is to simply take a union bound over all possibilities $L=\ell$
and $R=r$ with $\ell,r\leq n^3$, and then to note that it is very unlikely that $L>n^3$ or $R>n^3$.

We mention that we shared this argument with the authors of \cite{hoffman2026local}, who used
it with our permission in Section~4 of that article in advance of this present work.

\subsubsection*{Theorem~\ref{thm:ht}}
Consider the driven-dissipative chain on $\ii{1,n}$ with central driving.
To prove that the avalanche extent $X_n(\sigma)$ satisfies
$\P_{\pi_n}\bigl(X_n(\sigma)\geq m\bigr)=\omega(1/m)$, we note that if $\sigma_0\sim\pi_n$,
then $\sigma_s$ is given by executing $s$ avalanches in succession, all at stationarity.
If the extents of these avalanches are all bounded by $m$, then all $s$ particles added to $\config_0$ remain within distance $m$ of the center in $\config_s$. 
If $s= \epsilon m$, then $\config_s$ has density above $\critFE$ on an interval of length~$2m$, violating \thref{thm:regularity}. Thus we can deduce that one of the $s$ avalanches has extent greater than $m$ sites w.h.p.,
and hence a single avalanche must have extent greater than $m$  with probability at least $1/s=1/\epsilon m$.
\subsubsection*{Theorem~\ref{thm:remix}}

Consider the driven-dissipative ARW chain $(\sigma_t)_{t \geq 0}$ at stationarity with central
or uniform driving.
By the abelian property, we may reach $\config_{\epsilon n}$ by adding all $\floor{\epsilon n}$ driving
particles to $\sigma_0$ to obtain a configuration $\sprinkledconfig$,
which we then stabilize to get $\config_{\epsilon n}$.
To prove that $\config_{\epsilon n}$ has mixed given $\config_0$, it is sufficient
by \cite[Theorem 1.2]{hoffman2025cutoff} to show that every site of $\ii{1,n}$
is visited in the course of stabilizing $\tau$ with high probability.

For the central driving case,
since $\config_0$ and $\config_{\epsilon n}$ are both stationary, their densities are both close
to $\critFE$ by \eqref{eq:regularity}. Thus particles must be expelled when stabilizing $\tau$,
or else the density of $\config_{\epsilon n}$ would be greater than the density of $\config_0$
by $\epsilon$. Suppose without loss of generality that a particle exits at the left of the interval.
Then the entire left half of $\In$ is visited. Now we invoke \thref{lem:nucleate},
where we show that if we visit an entire interval $\ii{1, x}$ containing the center, 
then we will also visit $\ii{1,x+1}$ with high probability. Thus, the stabilization \emph{nucleates}:
by visiting $\ii{1,n/2}$, we ensure that we visit $\ii{1,n/2+1}$, thus ensuring
that we visit $\ii{1,n/2+2}$, and so on until the entire interval is visited.

For uniform driving, we use a similar nucleation argument, but it is more difficult
to show that a macroscopically large interval is visited to start the nucleation, since
a single particle exiting $\ii{1,n}$ no longer implies that an entire half-interval is visited.
But we know that in fact $\Omega(n)$ particles must exit the interval, which implies
that an interval of length $\Omega(n)$ on one of the boundaries is visited. We then
use this interval in place of the half interval to start the nucleation.

\section{Setup and notation} \label{sec:setup}
We use the standard sitewise construction of ARW as in \cite{rolla2020activated},
in which particles move by executing stacks of instructions at each site telling them
to jump or try to sleep.
We formally view a particle configuration on
some $V \subseteq \Z$ as a map $\config \colon \Z \to \N \cup \{\s\}$ that is identically
zero off of $V$, with $\config(x) = \s$ indicating the presence of a single sleeping particle at $x$ and $\config(x) = k \in \N$ indicating the presence of $k$ active particles. 
A configuration is \emph{stable} if it has no active particles, i.e., it takes only
the values $0$ and $\s$.
\emph{Toppling} a site is the act of executing the next instruction on the stack there.
A toppling is \emph{legal} (resp.\ \emph{acceptable}) if it occurs at a site with an active particle 
(resp.\ with a particle).
An odometer is a map $\Z\to\NN$ that we typically think of as giving counts of topplings at each site.
In particular, any finite sequence of topplings gives rise to an odometer counting the topplings at each site.
By the abelian property \cite[Lemma~2.4]{rolla2020activated},
all legal sequences of topplings within a set $U\subseteq V$ leaving a stable configuration on $U$ have
the same odometer
and hence produce the same configuration.
For any finite subset $U\subseteq V$, we let $\Stab[U]\! \sigma$ denote the configuration
on $U$ that agrees with this stable configuration on $U$; that is, particles
ejected from $U$ in the course of the stabilization are discarded, so that
$\abs{\Stab[U]\!\sigma}$ counts the particles remaining in $U$.
Note that $\Stab[U]\!\sigma$ depends only on the particles in $\sigma$ within $U$.
We may abbreviate $\Stab[U]\!\sigma$ to $\Stab \config$ when the domain of $\config$ is clear.
Note that the odometer stabilizing a configuration on a set $U$ takes
the value $0$ off $U$.

We write $\lt{f}{v}$ and $\rt{f}{v}$ to denote the number of left and right instructions,
respectively, executed by the odometer $f$ at $v$ (i.e., the number of instructions of the given type
within the first $f(v)$ instructions at site~$v$).
 As in \cite{rolla2020activated}, we define $\s + k := k + 1$ for all $k \ge 1$. We also let $\abs{\s} := 1$ so that $\abs{\config(x)}$ denotes the total number of particles at a given site $x$, and we consider
$\NN\cup\{\s\}$ to have the total ordering $0<\s<1<2<\cdots$.
We note the oft-used fact that given the starting configuration, 
the odometer of a sequence of topplings applied to it,
and the instruction stacks, we can determine the final configuration via mass-balance equations:
\begin{lemma}\thlabel{lem:mass.balance}
  Let $\sigma\colon \ZZ\to \NN\cup\{\s\}$ be a configuration, let $\tau$ be the configuration after
  applying a finite sequence of acceptable topplings to $\sigma$, and let $f$ be the odometer
  of the toppling sequence. Then
  \begin{align}\label{eq:mass.balance}
    \abs{\tau(v)} = \abs{\sigma(v)} + \rt{f}{v-1} + \lt{f}{v+1} - \lt{f}{v} - \rt{f}{v}.
  \end{align}
  If $\abs{\tau(v)}=1$, then $\tau(v)=\s$ if and only if either $f(v)\geq 1$
  and the $f(v)$th instruction executed at site~$v$ is a sleep instruction,
  or $f(v)=0$ and $\sigma(v)=\s$.
\end{lemma}
\begin{proof}
  Equation~\eqref{eq:mass.balance} holds by counting the inflow
  and outflow at site~$v$.
  
  Suppose $\abs{\tau(v)}=1$. If $\sigma(v)=\s$ and $f(v)=0$, then the inflow
  to $v$ is zero or we would have $\abs{\tau(v)}>1$, 
  and hence the sleeping particle initially at $v$ is untouched,
  yielding $\tau(v)=\s$. If $f(v)\geq 1$ and the $f(v)$th instruction executed there is a
  sleep instruction, then there is exactly one particle at $v$ when 
  the instruction is executed:
  if there is no particle present, then the toppling is unacceptable;
  if there are multiple particles present, then $\abs{\tau(v)}>1$ since no further topplings occur at $v$. Thus 
  there is a sleeping particle at $v$ after this toppling, and no other particle
  moves to $v$ after this since no particle leaves and $\abs{\tau(v)}=1$.
  
  Conversely, suppose $\tau(v)=\s$. If $f(v)=0$, then no particle can fall asleep at $v$,
  and hence a sleeping particle must have been present initially, proving $\sigma(v)=\s$.
  If $f(v)\geq 1$, then the final instruction executed at $v$ must be a sleep instruction:
  after a jump instruction, there would not be a sleeping particle present at $v$,
  and there will be no other opportunity to put a particle to sleep there.
\end{proof}

We note here that by a reformulation of the abelian property \cite[Lemma 1.3]{levine2021exact}, 
\begin{equation}\label{eq:sprinkle-early}
    \config_t = \Stab[V]\!\bigg(\config_0 + \sum_{s = 1}^{t}\delta_{u_s}\bigg)
\end{equation}
for the driven-dissipative Markov chain $(\config_t)_{t \ge 0}$ on $V$ with driving $\mathbf u$.

We say that a configuration $\config$ on a finite set $V \subset \Z$ \emph{visits} a site $x \in V$ if 
$x$ contains an active particle at any time during the stabilization of $\config$ on $V$ (or equivalently
if the odometer stabilizing $\sigma$ is positive at $x$). We say that $\config$ visits the subset $U \subseteq V$ if it visits every site of $U$. The \emph{preemptive abelian property} \cite[Corollary~3.2]{hoffman2025cutoff} states that if $\config$ visits $U$, then the stabilizing odometer for $\config$ is not altered
if we wake all particles in $U$ at the start. In particular, the final configuration for the system
and the set of sites visited are unaffected by preemptively waking the particles in $U$.
To state this property more precisely, define $\A[U]\config$
 as the configuration obtained from $\config$ by waking all particles on $U$.
\begin{lemma}[Preemptive abelian property, {\cite[Corollary~3.2]{hoffman2025cutoff}}]
    For a configuration $\config$ on a finite set $V \subset \Z$, if $\config$ visits $U \subseteq V$, then 
    the odometers stabilizing $\sigma$ and $\A[U]\!\config$ are the same.
\end{lemma}

We let $\1_U$ denote the configuration consisting of a single active particle
at each site in $U$. An important fact in our arguments
is that $\Stab[\ii{1,n}]\1_{\ii{1,n}}\sim\pi_n$ \cite[Theorem~2.1]{levine2021exact}.

\section{Proof of Theorem~\ref{thm:regularity}} \label{sec:regularity}
In Section~\ref{sec:sketches}, we sketched the idea
of viewing the configuration on $U\subseteq V$
after stabilization on $V$ as consisting of a random number
of steps in the driven-dissipative chain on $U$.
We begin with a sort of spatial Markov property for odometers that makes
this idea formal.
The property states that if $f$ is the odometer
stabilizing a configuration $\sigma$ on $V$, then $f\vert_U$ for any $U\subseteq V$ is the
odometer stabilizing $\sigma+\nu$ on $U$, where $\nu$ consists of all particles
sent into $U$ by $f$. See also \cite[Lemma~4.1]{hoffman2026local} for a similar statement.
We state the lemma only in the one-dimensional case though it holds
more generally. 

\begin{lemma}\thlabel{lem:restriction}
  Let $f$ be the odometer stabilizing a configuration $\sigma$ on $\ii{1,n}$.
  Then for any $1\leq a\leq b\leq n$, the odometer stabilizing 
  $\sigma+\rt{f}{a-1}\delta_a+\lt{f}{b+1}\delta_b$ on $\ii{a,b}$
  is $f\vert_{\ii{a,b}}$, and the configurations
  \begin{align*}
    \Stab[\ii{a,b}]{\bigl(\sigma+\rt{f}{a-1}\delta_a+\lt{f}{b+1}\delta_b\bigr)}
    \qquad\qquad\text{and}\qquad\qquad\Stab[\ii{1,n}]{\sigma}
  \end{align*}
  are equal on $\ii{a,b}$.
\end{lemma}
\begin{proof}
  Take a sequence of legal topplings stabilizing $\sigma$ on $\ii{1,n}$ and 
  consider the subsequence consisting of all topplings that are within $\ii{a,b}$,
  whose odometer is $f\vert_{\ii{a,b}}$.
  It will be convenient to view $f\vert_{\ii{a,b}}$ as having the domain $\ii{1,n}$ rather
  than $\ii{a,b}$, taking the value $0$ off of $\ii{a,b}$.
  We claim that this subsequence of topplings is legal for
  $\sigma+\rt{f}{a-1}\delta_a+\lt{f}{b+1}\delta_b$ and stabilizes it on $\ii{a,b}$,
  which proves that $f\vert_{\ii{a,b}}$ is its stabilizing odometer by the abelian property 
  \cite[Lemma~2.4]{rolla2020activated}.
  The idea is that if we execute the two toppling sequences in the two systems
  side by side, the only difference
  is that the particles sent into $\ii{a,b}$ in the original system
  are present in the subsequential system from the start, and so
  the topplings in the subsequential system remain legal. 
  
  Let $v_1,\ldots,v_N\in\ii{1,n}$ be the original toppling sequence,
  and let $v_{i_1},\ldots,v_{i_M}\in\ii{a,b}$ be the subsequence.
  Let $\sigma=\sigma_0,\ldots,\sigma_N$ denote the configurations
  given by applying the original toppling sequence to $\sigma$, i.e.,
  $\sigma_i=\Phi_{(v_1,\ldots,v_i)}\sigma$ in the notation of \cite{rolla2020activated}.
  Let $\sigma'_0,\ldots,\sigma'_M$
  be the configurations given by applying the subsequence of topplings to
  $\sigma+\rt{f}{a-1}\delta_a+\lt{f}{b+1}\delta_b$. To make $\sigma'_0,\ldots,\sigma'_M$
  well defined, if the toppling to be applied to $\sigma'_i$ is illegal,
  stop toppling and set $\sigma'_i=\cdots=\sigma'_M$.
  To prove that no illegal toppling occurs, we will show for all $1\leq k\leq M$ that
  \begin{align}\label{eq:moreparticles}
    \sigma'_{k-1}(v_{i_k})\geq \sigma_{i_k-1}(v_{i_k}).
  \end{align}
  Note that this inequality uses the order $0<\s<1<\cdots$.
  The left-hand side is the state of site $v_{i_k}$ immediately before it is toppled in the
  subsequential system, and the right-hand side is the same vertex's state before being
  toppled in the original system. 
  We have $\sigma_{i_k-1}(v_{i_k})\geq 1$ by the legality of the original toppling,
  and hence \eqref{eq:moreparticles} implies the legality of the subsequential toppling.
  
  To prove that \eqref{eq:moreparticles} holds for all $k$, 
  we assume inductively that $\sigma'_{j-1}(v_{i_j})\geq \sigma_{i_j-1}(v_{i_j})$
  for all $j<k$.
  Let $g$ be the odometer of the toppling sequence
  $v_1,\ldots,v_{i_k-1}$, which is legal for $\sigma_0=\sigma$
  and takes us from $\sigma_0$ to $\sigma_{i_k-1}$. 
  The toppling sequence $v_{i_1},\ldots,v_{i_{k-1}}$ is legal for
  $\sigma'_0$ by the inductive hypothesis and takes us from $\sigma'_0$
  to $\sigma'_{k-1}$, and its odometer is $g\vert_{\ii{a,b}}$.
  Let $w=v_{i_k}$, so that our goal is to show that 
  $\sigma'_{k-1}(w)\geq \sigma_{i_k-1}(w)$.
  Now, we apply \thref{lem:mass.balance} to determine that
 \begin{alignat}{6}
  \abs{\sigma_{i_k-1}(w)} &=  
   \abs{\sigma(w)}
    && {}+ \rt{g}{w-1}
    && {}+ \lt{g}{w+1}
    && {}- \lt{g}{w}
    && {}- \rt{g}{w},\label{eq:sigma}\\\intertext{and}
    \abs{\sigma'_{k-1}(w)} &= \abs{\sigma'_0(w)}
    && {}+ \rt{g\vert_{\ii{a,b}}}{w-1}
    && {}+ \lt{g\vert_{\ii{a,b}}}{w+1}
    && {}- \lt{g\vert_{\ii{a,b}}}{w}
    && {}- \rt{g\vert_{\ii{a,b}}}{w}.\label{eq:sigma'}
\end{alignat}
  We claim now that 
  \begin{align}\label{eq:abs.version}
    \abs{\sigma'_{k-1}(w)}\geq\abs{\sigma_{i_k-1}(w)}.
  \end{align}
  Analyzing \eqref{eq:sigma'} depending on whether $w\in\{a,b\}$, 
  the first term of its right-hand side is
  \begin{align}
    \begin{aligned}
    \abs{\sigma'_0(w)} &= \abs{\sigma(w)} + \1\{w=a\}\rt{f}{w-1}
      + \1\{w=b\}\lt{f}{w+1}\\
      &\geq \abs{\sigma(w)} + \1\{w=a\}\rt{g}{w-1}
      + \1\{w=b\}\lt{g}{w+1},
    \end{aligned}\label{eq:ie1}
  \end{align}
  since $f\geq g$ pointwise. The next two terms are
  \begin{align}
    \rt{g\vert_{\ii{a,b}}}{w-1}
    +\lt{g\vert_{\ii{a,b}}}{w+1} &= \1\{w\neq a\}\rt{g}{w-1}
     + \1\{w\neq b\}\lt{g}{w+1}.\label{eq:ie2}
  \end{align}
  And for the final two terms,
  \begin{align}
    - \lt{g\vert_{\ii{a,b}}}{w}
    - \rt{g\vert_{\ii{a,b}}}{w} = 
        - \lt{g}{w}
    - \rt{g}{w}.\label{eq:ie3}
  \end{align}
  Equations \eqref{eq:ie1}--\eqref{eq:ie3} combine to prove that
  \eqref{eq:sigma'} is at least as large as \eqref{eq:sigma}, proving \eqref{eq:abs.version}.
  
  It remains now to rule out the possibility that $\sigma'_{k-1}(w)=\s$.
  Suppose this holds. 
  By \thref{lem:mass.balance}, the last
  instruction executed at $w$ under $g\vert_{\ii{a,b}}$ is a sleep
  instruction, or we have $\sigma'_0(w)=\s$ and $g\vert_{\ii{a,b}}(w)=0$.
  In the first case, the last instruction executed at $w$ under $g$
  is the same sleep instruction, and $\abs{\sigma_{i_k-1}(w)}=1$ by \eqref{eq:abs.version}
  and legality of the original toppling sequence.
  Hence $\sigma_{i_k-1}(w)=\s$ by \thref{lem:mass.balance}, contradicting 
  the legality of the original toppling sequence.
  In the second case, since $\sigma'_0(w)=\s$
  and $\sigma'_0=\sigma+\rt{f}{a-1}\delta_a+\lt{f}{b+1}\delta_b$, we must have $\sigma(w)=\s$
  as well. Since $\abs{\sigma_{i_k-1}(w)}=1$ by \eqref{eq:abs.version}
  and legality of the original toppling sequence,
  and since $g(w)=g\vert_{\ii{a,b}}(w)=0$, \thref{lem:mass.balance} again
  implies that $\sigma_{i_k-1}(w)=\s$, a contradiction.
  Thus $\sigma'_{k-1}(w)\neq\s$, and therefore \eqref{eq:abs.version}
  implies \eqref{eq:moreparticles}, completing the induction and establishing
  that the subsequence is a legal toppling sequence for $\sigma'_0$.  
  
  Finally, we claim that the result of applying the subsequence of topplings to   
  $\sigma'_0$ is a configuration matching $\Stab[\ii{1,n}]\sigma$
  on $\ii{a,b}$. This claim completes the proof, since then
  the subsequence is stabilizing for $\sigma'_0=\sigma+\rt{f}{a-1}\delta_a+\lt{f}{b+1}\delta_b$ 
  on $\ii{a,b}$, produces the same configuration as $\Stab[\ii{1,n}]\sigma$
  on $\ii{a,b}$, and has odometer $f\vert_{\ii{a,b}}$.
  The argument is nearly the same as the induction we just carried out,
  with $f$ in the role of $g$.
  By \thref{lem:mass.balance}, for any $v\in\ii{a,b}$,
  \begin{alignat}{5}
  \abs{\sigma_N(v)}&= \abs{\sigma(v)}\label{eq:sigma.again}
    && {}+ \rt{f}{v-1}
    && {}+ \lt{f}{v+1}
    && {}- \lt{f}{v}
    && {}- \rt{f}{v}\\\intertext{and}
  \abs{\sigma'_M(v)} &= \abs{\sigma'_0(v)}\label{eq:sigma'.again}
    && {}+ \rt{f\vert_{\ii{a,b}}}{v-1}
    && {}+ \lt{f\vert_{\ii{a,b}}}{v+1}
    && {}- \lt{f\vert_{\ii{a,b}}}{v}
    && {}- \rt{f\vert_{\ii{a,b}}}{v}.
  \end{alignat}
  By the same reasoning as in the deduction of \eqref{eq:abs.version}
  from \eqref{eq:sigma} and \eqref{eq:sigma'}, we claim that
  $\abs{\sigma'_M(v)}=\abs{\sigma_N(v)}$. Indeed, the last two terms on the right-hand sides
  of \eqref{eq:sigma.again} and \eqref{eq:sigma'.again} are equal, and
  \begin{align*}
    \abs{\sigma'_0(v)} &= \abs{\sigma(v)} + \1\{v=a\}\rt{f}{v-1}
      + \1\{v=b\}\lt{f}{v+1}
  \end{align*}
  while
  \begin{align*}
    \rt{f\vert_{\ii{a,b}}}{v-1} + \lt{f\vert_{\ii{a,b}}}{v+1}
      &= \1\{v\neq a\}\rt{f}{v-1} + \1\{v\neq b\}\lt{f}{v+1}.
  \end{align*}
  Now it remains to prove that when $\abs{\sigma'_M(v)}=\abs{\sigma_N(v)}=1$, we have $\sigma'_M(v)=\s$
  and not $\sigma'_M(v)=1$.
  Since $\sigma_N$ is stable on $\ii{1,n}$ 
  and $\abs{\sigma_N(v)}=1$, we have $\sigma_N(v)=\s$.
  By \thref{lem:mass.balance}, the last instruction executed
  at $v$ under $f$ is a sleep instruction, or $\sigma(v)=\s$ and $f(v)=0$.
  In the first case, the last instruction executed at $v$ under $f\vert_{\ii{a,b}}$
  is the same sleep instruction, and hence $\sigma'_M(v)=\s$.
  In the second, we have $\abs{\sigma'_0(v)}\geq 1$ since $\sigma'_0$ is the same
  as $\sigma$ plus possible extra particles at $a$ and $b$. Since $f\vert_{\ii{a,b}}(v)=0$
  we have $\abs{\sigma'_M(v)}\geq\abs{\sigma'_0(v)}$. But $\abs{\sigma'_M(v)}=1$,
  implying that $\abs{\sigma'_0(v)}=1$ and therefore that $\sigma'_0$ has only
  the original particle from $\sigma$ at $v$, meaning that $\sigma'_0(v)=\s$.
  And now $\sigma'_0(v)=\s$ and $f\vert_{\ii{a,b}}(v)=0$ together imply
  that $\sigma'_M(v)=\s$ by \thref{lem:mass.balance}.
\end{proof}

\begin{proof}[Proof of \thref{thm:regularity}]
  Let $f$ be the odometer stabilizing the configuration $\1_{\ii{1,n}}$
  on $\ii{1,n}$.
  Write $I=\ii{a,b}$, and let $L=\rt{f}{a-1}$ and $R=\lt{f}{b+1}$.
  (Note that $L$ and $R$ are the number of left and right stabilizations
  as in the sketch in Section~\ref{sec:sketches}, by the abelian property.) 
  Let
  \begin{align*}
    \tau(\ell,r)=\Stab[\ii{a,b}]\bigl(\1_{\ii{a,b}} + \ell\delta_a + r\delta_b\bigr).
  \end{align*}
  The configuration $\tau(L,R)$ is equal to $\Stab[\ii{1,n}]\1_{\ii{1,n}}$ 
  on $\ii{a,b}$ by \thref{lem:restriction}.
  
  Now, we observe that for any deterministic nonnegative integers $\ell$ and $r$,
  we have $\tau(\ell,r)\sim \pi_{b-a+1}$.
  Indeed, $\Stab[\ii{a,b}]\1_{\ii{a,b}}\sim\pi_{b-a+1}$ by \cite[Theorem~2.1]{levine2021exact},
  and adding extra particles to $\1_{\ii{a,b}}$ does not affect the distribution of its
  stabilization, since by the abelian property and strong Markov property for instruction stacks \cite[Proposition~4]{levine2021source}, we can stabilize the configuration
  by first toppling the extra particles until they reach the sink, at which point we are reduced
  to stabilizing $\1_{\ii{a,b}}$. Hence, with $m=b-a+1$, we can apply \eqref{eq:regularity} to obtain
  \begin{align}\label{eq:regularity.stationary}
    \P\Bigl(  \abs[\big]{\critFE m  -  \abs{\tau(\ell,r)}} \geq \epsilon m \Bigr) \leq Ce^{-cm}
  \end{align}
  for any fixed $\ell,r\geq 0$, for constants $C,c$ depending on $\epsilon$ and $\lambda$ 
  but not on $m$ or $n$.
  Applying this estimate to all $0\leq \ell,r\leq n^3$ by a union bound,
  \begin{align*}
    \P\Bigl(\abs[\big]{\critFE m  -  \abs{\tau(L,R)}} \geq \epsilon m \Bigr) \leq Cn^6e^{-cm} + \P(\max(L,R)>n^3).
  \end{align*}
  The proof is completed now by bounding the tails of $L$ and $R$.
  By the least-action principle, the odometer $f$ is at most the odometer given
  by acceptably toppling all $n$ particles on $\ii{1,n}$ until they reach the sink.
  Hence a crude bound on $L$ is given by the number of jumps from $a-1$ to $a$
  of $n$ independent random walks started at the sites $1,\ldots,n$ and killed at the boundary.
  The local time for each random walk
  at $a-1$ (which bounds the number of jumps from $a-1$ to $a$) 
  is exponentially unlikely to exceed $n^2$, since it is stochastically
  dominated by $\Geo(1/n)$. And hence $L$ is exponentially unlikely to exceed $n^3$.
  The same is true for $R$, thus proving 
  \begin{align*}
    \P\Bigl(\abs[\big]{\critFE m  -  \abs{\tau(L,R)}} \geq \epsilon m \Bigr) \leq Cn^6e^{-cm} + C'e^{-c'n}
  \end{align*}
  for some constants $C',c'>0$,
  which after adjusting constants proves \eqref{eq:log-regularity}.
  (We note that using the tool of layer percolation defined in
  \cite{hoffman2024density}, one could show that
  $L$ and $R$ are exponentially unlikely to exceed some constant times $n^2$,
  thereby reducing the $n^6$ factor in \eqref{eq:log-regularity} to $n^4$.)
  Finally,  we note that for any $B > 6/c$ the right-hand side of \eqref{eq:log-regularity}
  vanishes as $n \to \infty$ whenever $m \geq B\log n$.  
\end{proof}

The trick in the previous theorem---applying a union bound over successive stabilizations
of a subinterval---also works outside of stationarity. For example, we can use it to prove that after stabilizing \emph{any} configuration of active particles, the density is unlikely to be high on a subinterval. We record the observation here though we do not need it in this paper.
\begin{prop}\thlabel{prop:upper.regularity}
  Fix $\epsilon,\lambda>0$. Let $\sigma$ be any configuration of active particles on $\ii{1,n}$.
  Then for any interval $I\subseteq\ii{1,n}$ of length~$m$
  \begin{align*}
    \P\Bigl(\abs[\big]{\bigl(\Stab[\ii{1,n}]\sigma\bigr)(I)}\geq (\critFE + \epsilon) m \Bigr) \leq Cn^6e^{-cm}
  \end{align*}
  for some constants $c,C>0$ depending on $\epsilon$ and $\lambda$ but not on $m$ or $n$.
\end{prop}
\begin{proof}
  First, it suffices to prove the result under the assumption that $\abs{\sigma}\leq n$,
  since we can successively topple sites containing multiple particles until
  all sites are empty or contain exactly one active particle, thus producing a configuration with at most $n$ particles with the same stabilization as $\sigma$. By the strong Markov property for instruction stacks \cite[Proposition~4]{levine2021source},
  we may condition on this configuration and apply the case $\abs{\sigma}\leq n$.
  
  We now define
  \begin{align*}
    \tau(\ell,r) = \Stab[\ii{a,b}]\bigl(\sigma+\ell\delta_a+r\delta_b\bigr)
  \end{align*}
  and proceed exactly as in the proof of \thref{thm:regularity}, but instead of arguing
  that $\tau(\ell,r)$ is at stationarity and applying \eqref{eq:regularity} 
  to obtain \eqref{eq:regularity.stationary}, we apply \cite[Theorem~8.4]{hoffman2024density},
  an upper bound on density that applies to the stabilization of any configuration of active particles,
  to obtain
  \begin{align*}
    \P\Bigl(  \abs[\big]{\tau(\ell,r)} \geq (\critFE+\epsilon) m \Bigr) \leq C'e^{-c'm}.
  \end{align*}
  The rest of the argument is the same, noting that our assumption that $\abs{\sigma}\leq n$
  is what allows the argument for bounding the tails of $L$ and $R$ to go through.
\end{proof}

\section{Proof of Theorem~\ref{thm:ht}} \label{sec:ht}

Fix a sequence $(m_n)$ with $m_n\leq n/2$ and $m_n/\log n\to\infty$; we aim to show \eqref{eq:av-star}.
Take $\epsilon > 0$, and let $(\config_t)_{t \ge 0}$ be the driven-dissipative Markov chain with central driving starting from stationary initial state $\config_0 \sim \pi_n$.
Let $s_n := \ceil{2\epsilon m_n}$, and let $M_n=\max\{X_n(\sigma_t)\colon0 \le t < s_n\}$.
By stationarity, we have $X_n(\sigma_t)\eqd X_n(\sigma)$ for all $t$.
 By a union bound,
\begin{align*}
    \P(M_n \ge m_n) \le \sum_{t=0}^{s_n-1}\P\bigl(X_n(\sigma_t) \ge m_n\bigr) 
        &= s_n \P_{\pi_n}\bigl(X_n(\sigma) \ge m_n\bigr)\\
        &\le (2\epsilon \mkern 1mu m_n+1) \P_{\pi_n}\bigl(X_n(\sigma) \ge m_n\bigr).
\end{align*}
Hence
\begin{align}\label{eq:av-union-bound}
  \P_{\pi_n}\bigl(X_n(\sigma) \ge m_n\bigr) &\geq \frac{\P(M_n \ge m_n)}{ 2\epsilon \mkern 1mu m_n+1}.
\end{align}

Now we argue that $\P(M_n \ge m_n)\to 1$.
Let $J_n$ be the centered ball
\begin{align*}
  J_n:=\ii[\big]{\ceil{n/2} - m_n+1, \ceil{n/2} + m_n} \subseteq \ii{1,n}.
\end{align*}
If $M_n < m_n$, then no sites outside of $J_n$ are visited in $s_n$ steps of the driven-dissipative
chain. Thus $\abs{\config_{s_n}(J_n)} = s_n + \abs{\config_0(J_n)}$, since all $s_n$ of the 
driving particles remain in $J_n$ at step~$s_n$ of the chain.
The density on $J_n$ therefore increases by at least $s_n/2m_n\geq\epsilon$ from configuration $\config_0$
to configuration $\config_{s_n}$. Hence, if $M_n<m_n$, then either $\config_{s_n}$ has density on $J_n$
at least $\critFE+\epsilon/2$, or $\config_0$ has density at most $\critFE-\epsilon/2$.
Both events have vanishing probability by \thref{thm:regularity}, since $\config_0$ and $\config_{s_n}$ 
are both distributed as $\pi_n$ and $\abs{J_n}=2m_n$ eventually exceeds $B\log n$ for the constant 
$B=B(\lambda,\epsilon/2)$ from the theorem. Thus, $M_n \ge m_n$ with probability tending to one, and so by \eqref{eq:av-union-bound}, 
$$
    \liminf_{n \to \infty} m_n\P_{\pi_n}\bigl(X_n(\sigma) \ge m_n\bigr) \ge \frac{1}{2\epsilon}.
$$
Letting $\epsilon \searrow 0$ we obtain \eqref{eq:av-star}.

\section{Proof of Theorem~\ref{thm:remix}} \label{sec:remix}
We say that a sequence of events $(E_n)$ holds \emph{with overwhelming probability} (w.o.p.)\ if there
exist constants $c,C > 0$ not depending on $n$ such that $\P(E_n) \geq 1 - Ce^{-cn}$ for all $n$. 
We allow $c$ and $C$ to depend on the sleep rate $\lambda$ as well as the constant
$\epsilon$ specified in the statement of \thref{thm:remix}. The intersection
of finitely many events (or polynomially many in $n$)  holding with overwhelming probability
also holds with overwhelming probability by a union bound.

At various points in this section, we will need to argue that
it is likely that a particle will be ejected from $\ii{1,n}$
in $\epsilon n$ steps of the driven-dissipative chain.
Let $Z_n(k)$ be the number of particles sent to the sink
in $k$ steps of the driven-dissipative ARW chain on $\ii{1,n}$,
with any driving sequence. By \eqref{eq:sprinkle-early}, we can equivalently
define $Z_n(k)$ as the number of particles sent to the sink when stabilizing
the configuration $\sprinkledconfig := \stationaryconfig + \sum_{t=1}^{k}\delta_{u_t}$
where $\stationaryconfig\sim\pi_n$ and $\mathbf{u}=(u_t)_{t\geq 1}$ is the driving sequence.
Since the chain is at stationarity, the expected number of particles in the configuration
is constant, and hence $\E Z_n(k) = k$. Similar to an argument in the proof
of \thref{thm:ht}, the following result applies
\eqref{eq:regularity} to obtain exponential concentration of $Z_n(k)$, so long as $k=\Omega(n)$:
\begin{lemma}\thlabel{lem:particle.loss}
  For any $0<\delta<\epsilon$, it holds w.o.p.\ (with the constants in the exponential
  bound depending on $\epsilon$ and $\delta$) that 
  \begin{align*}
    \abs[\big]{Z_n\bigl(\floor{\epsilon n}\bigr) - \epsilon n} \leq \delta n.
  \end{align*}
\end{lemma}
\begin{proof}
  According to the second definition of $Z_n(k)$, we have
  \begin{align}\label{eq:stabtau}
    \abs{\Stab\tau} = \abs{\sigma}  + \floor{\epsilon n} - Z_n\bigl(\floor{\epsilon n}\bigr).
  \end{align}
  Since $\Stab\tau$ and $\sigma$ are both distributed as $\pi_n$, both $\abs{\Stab\tau}$ and $\abs{\sigma}$
  are within $\delta n/3$ of $\critFE n$ w.o.p.
  by \eqref{eq:regularity}, which combined with \eqref{eq:stabtau} completes the proof.
\end{proof}

\subsection{Central driving}
Fix $\epsilon > 0$.
Denote the center vertex by $\xmid := \lceil n / 2 \rceil$. By \eqref{eq:sprinkle-early}, we can view $\config_{\epsilon n}$ as the stabilization of
$$\tau := \sigma + \floor{\epsilon n} \delta_{\xmid}.$$ 
Applying \cite[Theorem 1.2]{hoffman2025cutoff} conditionally
given $\sigma$, it follows that
  \begin{align}\label{eq:visit-bound}
    \tvnorm{\QuenchedP^{\stationaryconfig}(\sigma_{\epsilon n} \in \cdot \: ) -\pi_n} &\leq n^3 \QuenchedP^{\sigma}\bigl(\tau \text{ does not visit all of } \In\bigr)  + C'e^{-c' n} \text{ a.s.}
  \end{align}  
  for all $n \ge 1$ and some $C',c'>0$ not depending on $n$.
  (The $C'e^{-c'n}$ term is a bound on $\P(H\geq n^3)$ in \cite[eq.~(4)]{hoffman2025cutoff}
  that holds because $n$ random walks in $\ii{1,n}$ all visit the sink in $n^3$ steps w.o.p.\ 
  by the argument at the end of the proof of \thref{thm:regularity}.)
  By taking the expectation
  of \eqref{eq:visit-bound} with respect to 
  $\sigma\sim\pi_n$, the following proposition proves \eqref{eq:remix} with central driving.

\begin{prop} \thlabel{prop:visit}
  If $\sigma\sim\pi_n$, then $\tau$ visits $\ii{1,n}$ w.o.p.
\end{prop}

To prove this proposition, we will show that if all particles on a large enough interval
are activated, they are likely to visit a slightly larger interval, thus setting in motion
a chain reaction. For a configuration $\sigma$, say that 
\begin{itemize}
  \item $\config$ \emph{nucleates right} at $x \in \ii{1, n-1}$ if $\A[\ii{1, x}]\config$ visits $x + 1$;
  \item $\config$ \emph{nucleates left} at $y \in \ii{2, n}$ if $\A[\ii{y, n}]\config$ visits $y - 1$;
  \item $\config$ \emph{thoroughly nucleates} if it nucleates right at all $x \in \ii{\xmid, n-1}$ and nucleates left at all $y \in \ii{2, \xmid}$.
\end{itemize}
We will use a criterion for nucleation from \cite{hoffman2025cutoff}.
For a configuration $\config$ on $\ii{1,n}$ and $1 \le a \le b \le n$, define the weighted mass statistic 
\begin{align}\label{eq:weighted.mass}
    M_{a,b}(\config) := \sum_{x=a}^b x\abs{\config(x)}.
\end{align}
For any $\alpha > 0$, \cite[Proposition 4.1]{hoffman2025cutoff} states that there are constants $c,C > 0$ depending on $\lambda$ and $\alpha$ but not $x$ or $n$ such that
\begin{equation}\label{eq:com-implies-nucleation}
    M_{1,x}(\config) \geq \frac{(\critFE + \alpha)x^2}{2} \;\implies\; \P(\config \text{ nucleates right at } x) \geq 1 - Ce^{-cx}.
\end{equation}
An analogous criterion for left-nucleation is obtained by modifying \eqref{eq:weighted.mass}
to weight the mass at position $x$ by $n-x+1$ rather than $x$.

One should think of $(\critFE+\alpha)x^2/2$ as representing an idealized value of $M_{1,x}(\config)$
for a configuration $\config$ of uniform density $\critFE+\alpha$. The next lemma is a routine
if technical calculation confirming that if $\config$ is close to uniform density $\rho$, 
then $M_{1,x}(\config)$ will be close to $\rho x^2/2$.
\begin{lemma}\thlabel{lem:weighted.mass.regularity}
  Let $0<\beta\leq\gamma$, and suppose that a configuration $\sigma$ on $\In$
  has density at least $\rho$ on all subintervals of length at least $\beta n$.
  Then for all $x\geq \gamma n$ and $n\geq 4/\beta$,
  \begin{align}\label{eq:weighted.mass.regularity}
    M_{1,x}(\sigma) \geq \frac{(1-5\beta/\gamma)\rho x^2}{2}.
  \end{align}
\end{lemma}
\begin{proof}
Let $L=\ceil{\beta n}$. We will partition $\ii{1,x}$ into intervals of length~$L$
and compute the weighted mass on each subinterval.
Since $\sigma$ has density at least $\rho$ on $\ii{a,a+L-1}$, we have
$M_{a,a+L-1}(\sigma)\geq a\rho L$. Hence, with $K=\floor{x/L}$ and assuming that $x/L\geq 2$,
\begin{align*}
  M_{1,x}(\sigma) \geq \sum_{j=0}^{K-1}M_{jL+1, (j+1)L}(\sigma)\geq
  \sum_{j=0}^{K-1}\rho L^2j \geq \frac{\rho L^2(\frac{x}{L}-2)^2}{2}
  \geq \frac{\rho (x^2-4Lx)}{2}.
\end{align*}
We can in fact drop the assumption $x/L\geq 2$, since the final bound is negative in that case and the statement holds trivially. Now we have $4L\leq 4\beta n+4\leq 5\beta n$ since $n\geq 4/\beta$,
and hence from $x\geq \gamma n$ we obtain $4Lx\leq 5\beta x^2/\gamma$, which proves \eqref{eq:weighted.mass.regularity}.
\end{proof}

\thref{lem:weighted.mass.regularity} together with \eqref{eq:com-implies-nucleation}
easily imply that $\tau$ thoroughly nucleates.

\begin{lemma}\thlabel{lem:nucleate}
If $\sigma\sim\pi_n$, then $\sprinkledconfig $ thoroughly nucleates w.o.p.
\end{lemma}
\begin{proof}
  
Fix a small $\beta>0$, and apply \thref{thm:regularity} together
with a union bound over the $O(n^2)$ many intervals of length at least $\beta n$
to deduce that all of them have particle density at least $\critFE-\beta$ w.o.p.
By \thref{lem:weighted.mass.regularity} applied with $\gamma=1/2$, it holds for any $x\geq \xmid$ that 
$M_{1,x}(\sigma)\geq (1-10\beta)(\critFE-\beta)x^2/2\geq(\critFE-11\beta)x^2/2$ w.o.p.
Hence
\begin{align*}
  M_{1,x}(\tau) \geq \frac{(\critFE-11\beta)x^2}{2} + \floor{\epsilon n}\frac{n}{2}
  \geq \frac{(\critFE-11\beta+\epsilon/2)x^2}{2}\text{ w.o.p.},
\end{align*}
and choosing $\beta$ small enough relative to the given choice of $\epsilon$,
we apply \eqref{eq:com-implies-nucleation} and conclude
that $\tau$ nucleates right at all $x\in\ii{\xmid,n-1}$ w.o.p.
Using a nearly identical argument to prove left-nucleation
(the symmetry is imperfect when $n$ is even), we conclude
that $\tau$ thoroughly nucleates w.o.p.
\end{proof}

\begin{proof}[Proof of \thref{prop:visit}]
By \thref{lem:particle.loss}, at least one particle is sent to the sink
during the stabilization of $\sprinkledconfig$ w.o.p.
And by \thref{lem:nucleate}, the configuration $\tau$ thoroughly nucleates w.o.p.
Now, assume these two events hold, and we show that $\tau$ visits $\ii{1,n}$.
Suppose first that a particle reaches the sink via the left boundary. Then $\sprinkledconfig$ visits all of $\ii{1,\xmid}$, since the set of all visited sites must be a connected interval containing $\xmid$. Therefore $\sprinkledconfig$ and $\A[\ii{1,\xmid}]\sprinkledconfig$ visit the same sites by the preemptive abelian property.
Thus $\sprinkledconfig$ visits $\xmid+1$, since $\A[\ii{1,\xmid}]\sprinkledconfig$ visits $\xmid+1$
as a consequence of $\sprinkledconfig$ thoroughly nucleating.
Now $\sprinkledconfig$ visits all of $\ii{1,\xmid+1}$, and we can repeat the argument to show
that it visits $\xmid+2$. Continuing in this way, we conclude that $\sprinkledconfig$ visits all of $\In$. 
The same argument using left nucleation rather than right nucleation
handles the case where a particle reaches the right boundary.
\end{proof}

\subsection{Uniform driving} \label{sec:uniform}
Let $\sprinkledconfig := \stationaryconfig + \sum_{t=1}^{\floor{\epsilon n}}\delta_{u_t}$ 
for uniform driving $\mathbf{u}$.
As with central driving, \thref{thm:remix} under uniform driving follows
from the following proposition by \eqref{eq:visit-bound}.
\begin{prop}\thlabel{prop:uniform.visiting}
  If $\sigma\sim\pi_n$,
  then $\sprinkledconfig$ visits $\In$ w.o.p. 
\end{prop}
\begin{proof}

Fix $\epsilon>0$. We can assume that $\epsilon\leq 1$, since if $\stationaryconfig + \sum_{t=1}^{N_0}\delta_{u_t}$ visits $\ii{1,n}$, then $\stationaryconfig + \sum_{t=1}^{N}\delta_{u_t}$ visits $\ii{1,n}$ for all $N\geq N_0$.
Let $I_n^- := \ii{1, \epsilon n / 10}$ and $I_n^+ := \ii{n-\epsilon n/10 +1, n}$. We show that $\sprinkledconfig$ visits $\In$ when the following three events hold.

\begin{enumerate}[label = (\roman*)]
  \item\label{event:spillage} At least $\epsilon n/2$ particles exit $\ii{1,n}$
    during the stabilization of $\tau$.
  \item\label{event:nucleation} $\sprinkledconfig$ nucleates right at each $\epsilon n / 10 \le x \le n - 1$ and left at each $2 \le y \le n-\epsilon n / 10+1$.
  \item\label{event:initial-spread} Both $\big|\sprinkledconfig(I_n^-)\big|$ and $\big|\sprinkledconfig(I_n^+)\big|$ are smaller than $\epsilon n / 4$.
\end{enumerate}
Suppose that~\ref{event:spillage}--\ref{event:initial-spread} hold. If $\sprinkledconfig$ does not visit some $x \in I_n^-$, then only particles
from $\sprinkledconfig(\ii{1, x-1})$ can reach the left boundary during the stabilization of $\sprinkledconfig$, so fewer than $\epsilon n / 4$ do so by~\ref{event:initial-spread}, and similarly when $I_n^+$ is not fully visited. By~\ref{event:spillage} though, one of the boundaries must receive at least $\epsilon n / 4$ particles, so in fact either all of $I_n^+$ or all of $I_n^-$ is visited, and thus all of $\In$ by~\ref{event:nucleation}, arguing via the preemptive abelian property as in the proof of \thref{prop:visit}. Thus all that remains is to show that \ref{event:spillage}--\ref{event:initial-spread} occur w.o.p.

Event~\ref{event:spillage} holds w.o.p.\ by \thref{lem:particle.loss}.
Event~\ref{event:initial-spread} holds w.o.p.\ because $\stationaryconfig$ has at most one particle per site, and a $\Bin(\floor{\epsilon n}, p)$-distributed number of the uniformly added particles will be placed within distance $\epsilon n /10$ of each boundary for $p = \lfloor \epsilon n / 10 \rfloor / n \le \epsilon / 10$. Basic concentration estimates for the binomial distribution ensure that w.o.p.\ the number of additional particles added within $\epsilon n /10$ of each boundary vertex is no more than $\epsilon n /9$. Hence $\sprinkledconfig$ contains no more than $\epsilon n /10 + \epsilon n /9 < \epsilon n /4$ particles within distance $\epsilon n/10$ of each boundary w.o.p. 

Finally, for event~\ref{event:nucleation}, fix some $\beta>0$ to be specified shortly.
Using \thref{thm:regularity} on $\stationaryconfig$ and binomial concentration on the 
driving particles, we conclude that $\sprinkledconfig$ has density at least $\critFE + \epsilon/2$ on 
all intervals with length at least $\beta n$ w.o.p.
Applying \thref{lem:weighted.mass.regularity} with $\gamma=\epsilon/10$
and choosing $\beta$ small enough that $5\beta/\gamma<\epsilon/4$, then for all
$x\geq \epsilon n/10$ we have $M_{1,x}(\sprinkledconfig)\geq(\critFE+\epsilon/8)x^2/2$ w.o.p.\ (we use $\epsilon\leq 1$ here). Now \eqref{eq:com-implies-nucleation} implies that $\sprinkledconfig$
nucleates right for all $x \ge \epsilon n / 10$, and left nucleation follows by symmetry.
\end{proof}

\subsection{Extensions}\label{sec:mixing.extensions}
As we mentioned in the introduction, to test the robustness of \thref{thm:remix}
we can generalize the driving sequence or the starting configuration.
For generalizing the driving sequence, we expect it
would be straightforward to adapt the proof of \thref{prop:uniform.visiting} to prove that the chain 
remixes from stationarity in $o(n)$ time so long as the driving sequence places a positive 
fraction of the particles in the bulk of
$\ii{1,n}$, i.e., at a macroscopic distance from the boundary.
Some sort of requirement of bulk driving is necessary by
\thref{thm:ce}, which states that the remixing time from stationarity is $\Omega(n)$
with driving from site~$1$.

For generalizing the assumption of starting the chain at stationarity, the goal would be to prove that
the chain still mixes in $o(n)$ time when starting from a configuration $\sigma$ close everywhere
in density to $\critFE$.
For example, say that a sequence of configurations $(\sigma^n)_{n\geq 1}$ of sleeping particles on $\ii{1,n}$ 
is $\critFE$-regular if for any $\epsilon',\beta>0$, it holds for large enough $n$ that $\sigma^n$ 
has density at least $\critFE-\epsilon'$ on all intervals of length at least $\beta n$.
(For example, if one forms $\sigma^n$ by placing a sleeping particle at each site independently
with probability $\critFE$, then $(\sigma^n)_{n\geq 1}$ is almost surely $\critFE$-regular.)
Under central driving, the proof of \thref{prop:visit} can be adapted to prove mixing in $o(n)$
time starting from $\sigma^n$. The main change is that one cannot use stationarity to immediately 
deduce that particles leave the interval when stabilizing $\sigma^n+\floor{\epsilon n}\delta_{\ceil{n/2}}$
as is done in \thref{lem:particle.loss}.
Instead, one can let the $\floor{\epsilon n}$ particles spread to an interval.
This interval of active particles then serves as a base for a nucleation argument just
as  $\ii{1,\ceil{n/2}}$ or $\ii{\ceil{n/2},n}$ did in the proof of \thref{prop:visit}.

This argument fails for uniform driving, since the driving particles do not immediately
form a macroscopic interval to start the nucleation.
\begin{problem} \thlabel{q:remix}
    Prove that if $(\sigma^n)_{n \geq 1}$ is a {$\critFE$-regular} sequence of stable
    configurations, then the driven-dissipative ARW chain on $\ii{1,n}$ 
    with uniform driving mixes in $o(n)$ time started from $\sigma^n$.
\end{problem}

\subsection{Slow remixing under left driving}

In this section, we prove that the remixing time at stationarity with left driving
is $\Omega(n)$:
\begin{prop} \thlabel{thm:ce}
  There exist $\epsilon,\delta>0$ depending only on $\lambda$ such that with driving at site~$1$, 
    $$\E_{\pi_n} \norm{ \QuenchedP^{\stationaryconfig}(\sigma_{\epsilon n} \in \cdot \:) - \pi_n}_{\mathrm{TV}} >  \delta  $$
    for all $n$.
\end{prop}

We first state that the odometer of an avalanche at stationarity initiated at $x$
is the same in expectation at each site as the odometer of a simple random walk from $x$ killed at the sink.
Similar ideas are exploited in  \cite{StaufferTaggi18, Taggi19, levine2021exact,junge2025mean}. 

Let $\odom_{\mathsf{Av}, x}$ be the odometer stabilizing the configuration $\config+\delta_x$ where
$\config\sim\pi_n$, i.e., the odometer of an avalanche starting at $x$ at stationarity. 
Let $\odom_{\mathsf{RW}, x}$ denote the odometer given by acceptably toppling
a single particle at $x$ until it leaves $\ii{1,n}$.
Recall that for $i\in\{\Left,\Right\}$, we write $f^i(v)$ to denote the number of instructions of
type $i$ executed by the odometer at site~$v$.

\begin{prop}\thlabel{lem:exp-odom}
    For any sites $x,y \in \In$ and instruction type $i\in\{\Left,\Right\}$, 
    \begin{align}\label{eq:exp-odom}
        \E \big[\odom_{\mathsf{Av}, x}^i(y)\big] = \E \big[\odom_{\mathsf{RW}, x}^i(y)\big].
    \end{align}
\end{prop}

\begin{proof}
    Consider the odometer $\odom$ stabilizing the configuration $\onepersite_{\In} + \delta_x$. As in \cite{levine2021exact}, we consider two ways of arriving at $\odom$. 
    Using procedure~A, we have $f=f_1+f_2$ where $f_1$ is the odometer given by    
    toppling the extra particle at $x$ until it exits the graph, and $f_2$ is the odometer
    counting the topplings from that point on to stabilize the remaining particles. 
    Alternatively, using procedure~B, we have
    $f = f'_1+f'_2$ where $f'_1$ is the odometer stabilizing $\onepersite_{\In}$
    and $f'_2$ is the odometer stabilizing the resulting configuration $\sigma+\delta_x$
    with $\sigma\sim\pi_n$.
    We claim that
    \begin{align}
      f_1&\eqd \odom_{\mathsf{RW}, x},\label{eq:f1}\\
      f'_2&\eqd \odom_{\mathsf{Av}, x},\label{eq:f'2}\\\intertext{and}
      f_2&\eqd f'_1.\label{eq:f2}
    \end{align}
    Indeed, \eqref{eq:f1} is evident. Equation~\eqref{eq:f'2} holds because after executing
    odometer $f_1'$ in procedure~B, we are left with a configuration $\sigma\sim\pi_n$ 
    \cite[Theorem~2.1]{levine2021exact},
    plus an additional particle at $x$. Equation~\eqref{eq:f2} holds because $f_2$ and $f'_1$
    are both odometers stabilizing $\onepersite_{\In}$. The odometer $f'_1$ carries this out
    with the original instruction stacks in procedure~B, while $f_2$ carries this
    out using different but identically distributed instructions by the strong Markov
    property for instruction stacks \cite[Proposition~4]{levine2021source}.
    We note that $f_1$ and $f_2$ are independent while $f'_1$ and $f'_2$ are not, though
    it is irrelevant for this argument.
    
    Now, we have $f_1^i(y)+f_2^i(y)={f_1'}^i(y) + {f_2'}^i(y)$, and we take expectations,
    apply \eqref{eq:f1}--\eqref{eq:f2}, and subtract off $\E {f'_1}^i(y) = \E f_2^i(y)$
     to obtain \eqref{eq:exp-odom}.
\end{proof}

We get the following corollary for controlling which sites are visited by an avalanche. Let $G_n(x,y)$ for $x,y \in \In$ denote the Green's function, specifying the expected number of visits to $y$ before leaving $\In$ of a random walk started from $x$.
\begin{cor}\thlabel{cor:exp-visits}
The expected number of times a particle jumps to $y \in \In$ during a stationary avalanche from $x\neq y$ is given by $G_n(x,y)$. 
\end{cor}
\begin{proof}
    This follows from \thref{lem:exp-odom} by summing over all neighbors $z$ of $y$ and instruction types sending a particle from $z$ to $y$.
\end{proof}

To prove \thref{thm:ce}, we will apply \thref{cor:exp-visits} to show that in $\floor{\epsilon n}$
steps of driving from the left starting from $\sigma\sim\pi_n$, it is unlikely that site~$n$ is visited.
Thus $\sigma_{\epsilon n}(n)$ is nearly deterministic given $\sigma$. Therefore the distribution of
$\sigma_{\epsilon n}$ given $\sigma$ cannot be close to $\pi_n$
so long as $\pi_n$ is not deterministic at site~$n$,
a fact we confirm now.

\begin{lemma}\thlabel{lem:endpoint-bounds}
    For all $n \ge 1$,
    $$
        \frac{\lambda}{1 + \lambda} \le \P_{\pi_n}\big(\config(n) = \s\big) \le 1 - \frac{1/2}{1 + \lambda}.
    $$
\end{lemma}
\begin{proof}
  We sample from $\pi_n$ by stabilizing $\onepersite_{\In}$.
  Carry out the stabilization by first weakly stabilizing at $n$, i.e.,
  toppling until the configuration has one active particle at $n$ and is stable on $\ii{1,n-1}$.
  Then topple the particle at $n$ once. 
  With probability $\frac{\lambda}{1 + \lambda}$ it immediately falls asleep, ending the stabilization with $\config(n) = \s$ and thus proving the lower bound (this is the same argument used in \cite{StaufferTaggi18} to show that $\critFE(d) \ge \smash{\frac{\lambda}{1 + \lambda}}$ in any dimension). On the other hand, with probability $\smash{\frac{1/2}{1 + \lambda}}$ it will jump to the right sink $n + 1$, ending the stabilization with $\config(n) = 0$, which proves the upper bound.
\end{proof}

\begin{proof}[Proof of \thref{thm:ce}] 
Let $q^* := \min(\frac{\lambda}{1 + \lambda}, \frac{1/2}{1 + \lambda})$, so that by \thref{lem:endpoint-bounds} we have
\begin{align}\label{eq:pibound}
    \min\bigl(\P_{\pi_n}(\config(n) = 0),\, \P_{\pi_n}(\config(n) = \s)\bigr) \ge q^* > 0.
\end{align}
By definition of the total variation norm, 
for a deterministic stable starting configuration $\config_0$ on $\In$ we have
\begin{align*} 
        \left\| \QuenchedP^{\config_0}(\config_{\epsilon n} \in \cdot \:) - \pi_n \right\|_{\mathrm{TV}} &\ge 
        \P_{\pi_n}\bigl(\config(n) \neq \config_0(n)\bigr)-\QuenchedP^{\config_0}\big(\config_{\epsilon n}(n) \neq \config_0(n)\big)\\
        &\geq q^* - \QuenchedP^{\config_0}\big(\config_{\epsilon n}(n) \neq \config_0(n)\big),
    \end{align*}
    applying \eqref{eq:pibound} in the second line.
    Replacing $\sigma_0$ with $\sigma\sim\pi_n$ and taking expectations gives
    \begin{align}\label{eq:tv-dist}
        \E_{\pi_n}\left\| \QuenchedP^\config(\config_{\epsilon n} \in \cdot \:) - \pi_n \right\|_{\mathrm{TV}} 
            &\ge q^* - \E_{\pi_n}\QuenchedP^\config\big(\config_{\epsilon n}(n) \neq \config(n)\big) \nonumber \\
            &\ge q^* - \sum_{k = 1}^{\floor{\epsilon n}}\E_{\pi_n}\QuenchedP^\config\big(\config_k(n) \neq \config_{k-1}(n)\big) \nonumber \\
            &= q^* - \floor{\epsilon n}\,\E_{\pi_n}\QuenchedP^\config\big(\config_1(n) \neq \config(n)\big) \nonumber \\
            &\ge q^* - \epsilon n\,\E_{\pi_n}\QuenchedP^\config\big(\config_1(n) \neq \config(n)\big).
    \end{align}
    The equality in the second-to-last line is due to stationarity. 
    
    Now $\config_1$ is the result of a stationary avalanche from the left endpoint. By \thref{cor:exp-visits} and Markov's inequality, this avalanche visits $n$ with probability at most $G_n(1,n)$; if the avalanche
    does not visit $n$, then $\config(n)=\config_1(n)$. Thus
    \begin{align*}
      \E_{\pi_n}\left\| \QuenchedP^\config(\config_{\epsilon n} \in \cdot \:) - \pi_n \right\|_{\mathrm{TV}}
        &\geq q^* - \epsilon n G_n(1,n).
    \end{align*}
     And $G_n(1,n) \le 2/n$, since by gambler's ruin the random walk reaches $n$ with exactly $1/n$ probability, and every visit has at least a $1/2$ probability of being the last since the particle may jump right. Thus, choosing any $\epsilon < q^* / 2$, we have by \eqref{eq:tv-dist} that
    \begin{align*}
        \E_{\pi_n}\left\| \QuenchedP^\config(\config_{\epsilon n} \in \cdot \:) - \pi_n \right\|_{\mathrm{TV}} 
            &\ge q^* - \epsilon n (2/n) = q^* - 2\epsilon > 0.\qedhere
    \end{align*}
\end{proof}

\section*{Artificial intelligence disclosure}
All the mathematical content and writing in this paper was carried out entirely by us, though
we made edits to the paper based on a thorough proofreading by Claude Opus~5 and 5.5.

\section*{Acknowledgments}

T.J.\ was partially supported by NSF grant DMS-2503779.
M.J.\ was partially supported by NSF grant DMS-2238272.

\bibliographystyle{amsalpha}
\bibliography{main}

\end{document}